\documentclass[12pt]{amsart}
\usepackage[margin=1.25in]{geometry}
\usepackage[toc,page]{appendix}
\usepackage{amsmath, amssymb}
\usepackage{hyperref}
\usepackage{cleveref}
\usepackage{tikz, ytableau}
\usepackage[backend=bibtex,natbib=true,style=alphabetic,maxnames=99]{biblatex}
\title{Maximal transitivity of the cactus group on standard Young tableaux}
\author{Sophia Liao and Leonid Rybnikov}

\newcommand{\syt}{\text{SYT}}

\newtheorem{thm}{Theorem}[section]
\newtheorem{prop}{Proposition}[section]
\newtheorem{cor}{Corollary}[section]

\newtheorem{lemma}{Lemma}[section]

\theoremstyle{definition}
\newtheorem*{defn*}{Definition}
\newtheorem{example}{Example}[section]

\begin{document}

\maketitle

\begin{abstract} The action of the cactus group $C_n$ on Young tableaux of a given shape $\lambda$ goes back to Berenstein and Kirillov and arises naturally in the study of crystal bases and quantum integrable systems. We show that this action is $2$-transitive on standard Young tableaux of the shape $\lambda$ if and only if $\lambda$ is not self-transpose and not a single hook. Moreover, we show that in these cases, the image of the cactus group in the permutation group of standard Young tableaux is either the whole permutation group or the alternating group, and prove that both cases are possible for infinitely many $\lambda$ (though the alternating group is more frequent). As an application, this implies that the Galois group of solutions to the Bethe ansatz in the Gaudin model attached to the Lie group $GL_d$ is, in many cases, at least the alternating group. This also extends the results of Sottile and White on the multiple transitivity of the Galois group of Schubert calculus problems in Grassmannians to many new cases.
\end{abstract}

\section{Introduction}\label{se:intro}

\subsection{Cactus group} The \emph{cactus group} \( C_n \) arises naturally in the study of crystal bases, combinatorics of tableaux, and quantum integrable systems. It can be defined in several equivalent ways. One definition presents \( C_n \) as the group generated by involutions \( s_{[i,j]} \) for all intervals \( [i,j] \subseteq [1,n] \), subject to the relations:
\begin{enumerate}
    \item \( s_{[i,j]}^2 = 1 \),
    \item \( s_{[i,j]} s_{[k,l]} = s_{[k,l]} s_{[i,j]} \) if \( [i,j] \cap [k,l] = \emptyset \),
    \item \( s_{[i,j]} s_{[k,l]} s_{[i,j]} = s_{[i+j-l,i+j-k]} \) if \( [k,l] \subseteq [i,j] \).
\end{enumerate}
An equivalent definition realizes \( C_n \) as the \emph{$S_n$-equivariant fundamental group} of the real locus \( \overline{M}_{0,n+1}(\mathbb{R}) \)  of the Deligne-Mumford moduli space $\overline{M}_{0,n+1}$ of stable genus 0 curves with \( n+1 \) marked points. 

More generally, the \emph{W-cactus group} \( \mathrm{Cact}_W \) was introduced by Davis, Januszkiewicz and Scott \cite{Davis2003Fundamental} as the fundamental group of an iterated blow-up of the projective hyperplane arrangement associated with a finite Coxeter group \( W \) and later studied by Henriques and Kamnitzer as a ``crystal limit'' of the Artin braid group $Br_W$ (see \cite{HK} for type A and \cite{Halacheva_Kamnitzer_Rybnikov_Weekes_2020,Halacheva2020,Losev2019,Bonnafe2016} for arbitrary type). It is an abstract group generated by involutions \( s_I \) corresponding to all connected subdiagrams \( I \subseteq S \) of the Coxeter diagram of \( W \), where \( S \) is the set of simple reflections. The group relations are: (1) \( s_I^2 = 1 \); (2) if \( I \) and \( J \) are disconnected, then \( s_I s_J = s_J s_I \); and (3) if \( J \subset I \), then \( s_I s_J s_I = s_{w_I(J)} \), where \( w_I \) is the longest element in the parabolic subgroup \( W_I \) and acts by diagram automorphisms on subsets. This generalizes the classical cactus group (which corresponds to \( W = S_n \)), and captures the symmetries arising in the category of crystals for reductive Lie algebras. In particular, \( \mathrm{Cact}_W \) acts on the crystal graph of an irreducible representation of \( \mathfrak{g} \) with Weyl group \( W \), with each \( s_I \) acting as a partial Schützenberger involution on the crystal restricted to weight strings associated with \( W_I \).

In \cite{Losev2019} and \cite{Bonnafe2016}, Losev and Bonnaf\'e constructed a $\textrm{Cact}_W$-action on one-side Kazhdan-Lusztig cells in $W$. For $W$ being a Weyl group of a finite-dimensional semisimple Lie algebra $\mathfrak{g}$, this action reflects the structure of the category \( \mathcal{O} \) for $\mathfrak{g}$ and intertwines with the duality and wall-crossing functors, making the cactus action categorical in nature.

\subsection{$C_n$-action on Standard Young Tableaux}  For $W=S_n$, the Kazhdan-Lusztig cells are identified with sets of \emph{standard Young tableaux} of a given shape $\lambda$ of size \( n \) via the Robinson–Schensted correspondence. Explicitly, each generator \( s_{[i,j]} \) acts by reversing the order of the entries in the subword \( w_i, w_{i+1}, \dots, w_j \) of a permutation \( w \in S_n \), interpreted in terms of the corresponding standard tableau. Equivalently, the generator \( s_{[i,j]} \) acts as the Sch\"utzenberger involution of the (skew) subtableau formed by the boxes containing thenumbers \(i,i+1,\ldots,j\).

Following Berenstein and Kirillov \cite{Berenstein_Kirillov_1996}, Chmutov, Glick, and Pylyavskyy \cite{Chmutov_Glick_Pylyavskyy_2020} considered the more general action of $C_n$ on the set of \emph{semistandard Young tableaux} of shape \( \lambda \), which label basis elements in the crystal \( B(\lambda) \) for \( \mathfrak{sl}_n \). They expressed the action of \( s_{[i,j]} \in C_n \) in terms of simpler \emph{Bender-Knuth} involutions $t_1,\ldots,t_{n-1}$ and showed that Bender-Knuth involutions can be expressed in terms of \( s_{[i,j]}\). For the $C_n$ action on standard Young tableaux, the involution $t_i$ switches $i$ with $i+1$ if this keeps the table standard (and does nothing otherwise). The observation that Bender-Knuth involutions generate the same group of transformations of standard Young tableaux as the Sch\"utzenberger involutions makes this group action much easier to explore.

\subsection{Main results} A group action is called $k$-transitive if, for any two ordered sets of $k$ distinct elements from the set being acted upon, there exists a group element that maps the first set to the second set. The main question we address in this paper is what is the maximal $k$ such that $C_n$-action on standard Young tableaux of the given shape is $k$-transitive. 

Let $\lambda$ be a partition of $n$ (i.e. a Young diagram with $n$ boxes). Denote by $\syt(\lambda)$ the set of all standard tableaux of the shape $\lambda$. The main results of the present paper are the following two statements. 

\medskip

\noindent \textbf{Theorem A.} Let $\lambda$ be a partition of $n$. \emph{\begin{enumerate}
\item The action of $C_n$ on $\syt(\lambda)$ is always transitive. 
\item This action is double transitive unless $\lambda$ is hook-shaped or self-transpose. \item The only $C_n$-invariant of pairs of standard Young tableaux of a hook-shape $\lambda$ is the cardinality of the set of numbers contained in the first row of both diagrams. 
\item The only invariant of pairs of tableaux for self-transpose $\lambda$ is whether the elements of the pair are different or not up to transposition.
\end{enumerate}}

\medskip

\noindent \textbf{Remark.} In fact, we prove a stronger statement on the transitivity of $C_n$ action on pairs of standard Young tableaux of \emph{different} shapes, see Theorem~\ref{2-trans}, and get the above $2$-transitivity result as a consequence (Corollary~\ref{co:2-trans}).

\medskip

\noindent \textbf{Theorem B.} \emph{Suppose that $\lambda$ is not hook-shaped and not self-transpose. Let $N$ be the cardinality of $\syt(\lambda)$. Then the image of $C_n$ in the permutation group $S_N$ of the set $\syt(\lambda)$ is either the whole $S_N$ or the alternating group $A_N$.}

\medskip

The strategy of the proof is, first, to show that the action of $C_n$ on $\syt(\lambda)$ is 3-transitive unless $\lambda$ is hook-shaped or self-transpose (similarly to $2$-transitivity, we in fact prove a stronger statement about transitivity of the $C_n$ action on triples of Young tableaux of \emph{different} shapes, see Theorem~\ref{th:3-trans} in the text). Next, according to \cite{Cameron_book}, this leaves the following three possibilities apart from $A_N$ or $S_N$:
\begin{enumerate}
    \item $\syt(\lambda)=\mathbb{F}_2^d$, an affine space over the $2$-element field $\mathbb{F}_2$, and the image of $C_n$ is (a subgroup of) the group $AGL(d,2)$ of affine transformations of this space.
    \item $\syt(\lambda)=\mathbb{F}_q\mathbb{P}^1$, the projective line over a finite field $\mathbb{F}_q$, and the image of $C_n$ is the group of projective transformations of the line $PGL(2,q)$.
    \item Sporadic examples of Mathieu groups $M_N$ acting on the $N$-element set, for $N=11,12,22,23,24$. 
\end{enumerate}
In all of the above examples, it is known how many fixed points may have involutions from the group. On the other hand, we can produce conjugacy classes of involutions in the image of the cactus group $C_n$ that have different numbers of fixed points, by composing independent Bender-Knuth involutions. This allows to exclude all the above cases.

\medskip

\noindent \textbf{Remark.} Both $S_N$ and $A_N$ are possible. The numerical experiments show that the whole symmetric group appears more frequently for smaller $n$, but occurrences of $A_N$ catch up as $n$ gets bigger. We expect that the cases where we have $A_N$ dominate when $n\to\infty$, though there are infinitely many cases where we get $S_N$. We discuss this in Section~\ref{se:further} of the paper, giving some sufficient condition for the image of the cactus group being $A_N$, and in the Appendix, giving the results of numerical experiments.

\subsection{Application to Galois group of Bethe ansatz equations for the $GL(N)$ Gaudin model} 
In \cite{Halacheva_Kamnitzer_Rybnikov_Weekes_2020}, Halacheva, Kamnitzer, Weekes, and the second author interpreted the action $C_n$-action on a Kashiwara crystal as a \emph{monodromy action} on solutions to Bethe ansatz in the Gaudin model (that is a completely integrable quantum spin chain arising as a degeneration of the XXX Heisenberg model).
This monodromy action can be viewed as a part of the \emph{Galois group action} on the set of solutions to the \emph{Bethe ansatz equations} in the Gaudin model attached to the Lie group $GL_d$. The Bethe ansatz equations depend on a collection of Young diagrams $\lambda^{(1)},\ldots,\lambda^{(n)}$ and $\mu$ of the height not greater than $d$. It is known (see e.g. \cite{MTV_symmetric_gp}) that once all $\lambda^{(i)}=(1)$, the solutions of Bethe ansatz are one-to-one with $\syt(\mu)$. Next, according to \cite{Halacheva_Kamnitzer_Rybnikov_Weekes_2020}, the natural monodromy $C_n$-action on this set is isomorphic to the above $C_n$-action on $\syt(\mu)$. So, our Theorem~B implies the following:

\medskip

\noindent \textbf{Theorem C.} \emph{Let the weights $\lambda^{(1)},\ldots,\lambda^{(n)}$ and $\mu$ be as follows:
\begin{itemize}
    \item $\mu$ is a non-hook-shaped and not self-transpose partition of $n$;
    \item  for all $i=1,\ldots,n$, $\lambda^{(i)}=(1)$ i.e. one-box Young diagrams.
\end{itemize} 
The Galois group of solutions to Bethe ansatz equations for the Gaudin model corresponding to $\lambda^{(1)},\ldots,\lambda^{(n)},\mu$ is at least the alternating group.}

\subsection{Application to Galois group of the Schubert calculus in Grassmannians} As shown by Mukhin, Tarasov, and Varchenko \cite{MTV_Schubert,MTV_symmetric_gp}, these equations correspond to special instances of \emph{Schubert calculus problems} on Grassmannians, more precisely, to the intersection points of real osculating Schubert varieties, see also \cite{White_2018,Brochier_Gordon_White_2023}. The Galois group acts naturally on the set of Bethe solutions, and the image of the cactus group can be identified with a subgroup there. 

This relates our result to the general problem of computing Galois groups of Schubert problems in Grassmannians, studied by Sottile and White in \cite{SW}. The \emph{Galois group} of a Schubert problem describes how the transversal $0$-dimensional intersections of Schubert varieties vary as the defining flags change continuously. Algebraically, one studies the function field of the parameter space of flags and the splitting field of the solutions to the Schubert problem over it. The Galois group of this field extension acts on the solutions by permutation, revealing how they move as parameters change around various loops in the parameter space of flags.

In \cite{SW}, Sottile and White showed that, in many cases, the above Galois group acts doubly transitively on the set of solutions, and in some cases, it is possible to see that this Galois group is at least the alternating group. Theorem~C extends the results of Sottile and White on the multiple transitivity of Galois group of Schubert calculus on Grassmannians to many new cases. Namely, let $\mu\vdash n$ be a non-hook-shaped and not self-transpose, of the height not greater than $d$ and  let $\lambda^{(0)}$ be the Young diagram obtained as the complement of $\mu$ in the $d\times (2n-d)$ rectangle. 

\medskip

\noindent \textbf{Theorem D.} \emph{Consider the Schubert problem in the Grassmannian $\textrm{Gr}(d,2n)$ of $d$-dimensional planes in the $2n$-dimensional vector space for the collection of partitions $\lambda^{(0)},\lambda^{(1)},\ldots,\lambda^{(n)}$ with $\lambda^{(0)}$ being as above and all other $\lambda^{(i)}=(1)$, i.e. one-box Young diagrams. The Galois group of this Schubert problem is at least the alternating group.}


\subsection{The paper is organized as follows} Section~\ref{se:prelim} gives the preliminaries on the cactus group and its action on Young tableaux following \cite{Chmutov_Glick_Pylyavskyy_2020}. The proof of Theorem~A occupies Sections~\ref{se:hook-and-symmetric}~and~\ref{se:2-trans}: in Section~\ref{se:hook-and-symmetric}, we describe the invariants of pairs of standard Young tableaux which are hook-shaped or self-transpose, and in Section~\ref{se:2-trans}, we prove $2$-transitivity in all other cases. The proof of Theorem~B occupies Sections~\ref{se:3-trans}~and~\ref{se:thm-B}: in Section~\ref{se:3-trans}, we prove $3$-transitivity, and in Section~\ref{se:thm-B}, we exclude all the cases except $A_N$ and $S_N$ by analyzing possible numbers of fixed points for involutions. In Section~\ref{se:sottile}, we prove Theorems~C~and~D. Finally, in Section~\ref{se:further}, we discuss the question of whether $A_N$ or $S_N$ occurs more frequently as the image of the cactus group, presenting partial results and numerical experiments. We summarize the results of the numerical experiments in Appendix A.

\subsection{Acknowledgements} Both authors thank the program MIT PRIMES, under which this project was started. Especially, we thank Tanya Khovanova and Thomas R\"ud for feedback on an early draft of the paper. We thank Yakob Kahane for suggesting the idea and the key reference for Proposition~\ref{pr:|syt|=even}. We are grateful to Joel Kamnitzer for suggetions and for providing us with some computational data, and to Alexander Postnikov for suggestions and references. We thank Son Nguyen, Pavlo Pylyavskyy and Frank Sottile for their valuable comments on the first arXiv version of the paper. L.R. thanks the Courtois Foundation for financial support.

\section{Cactus group and its action on standard Young tableaux}\label{se:prelim}

\subsection{Sch\"utzenberger involutions} 
The \emph{Schützenberger involution} \( \xi \) is an involution on the set $\syt(\lambda)$ of all standard Young tableaux of the shape $\lambda$. 

The algorithm makes use of \emph{jeu de taquin} slides and is as follows  (see, e.g. \cite{LY} for the details):

\begin{enumerate}
    \item Remove the box in the first row and first column of the tableau. Record the value \( m \) of this box.
    \item Perform a \emph{jeu de taquin} slide to fill the empty position.
    \item Place a new box in the last position to be emptied in the slide, labeled \( -m \).
    \item Repeat the process until all of the original boxes of \( T \) have been removed. 
\end{enumerate}

After all boxes of the original tableau have been removed in this way, the result is a tableau of the same shape, but with all negative entries. Adding a constant equal to the number of boxes in \( T \) plus one to each entry will yield the standard tableau \( \xi(T) \).

\subsection{Cactus group action on $\syt(\lambda)$} From the defining relations of $C_n$, it is seen that the involutions $s_{[1i]}$ for $i=2,\ldots,n$ generate $C_n$. The following is well-known (see, e.g. \cite{LY}):

\begin{prop}\label{pr:Schutzenberger}\begin{enumerate}
    \item For any $\lambda\vdash n$, there is an action of $C_n$ on $\syt(\lambda)$ taking $s_{[1i]}$ to the Sch\"utzenberger involution of the subtableau formed by the boxes that contain the numbers from $1$ to $i$.
    \item Under the above $C_n$-action on $\syt(\lambda)$, the involution $s_{[ij]}$ changes only the (skew) subtableau formed by the boxes containing the numbers from $i$ to $j$.
\end{enumerate}
\end{prop}

\subsection{Bender--Knuth Involutions}
For any $i=1,\ldots,n-1$, the \emph{Bender-Knuth involution} $t_i$ on $\syt(\lambda)$ interchanges the occurrences of $i$ and $i+1$ the resulting table is a valid standard Young tableau (and does nothing otherwise).

There is an explicit correspondence between Sch\"utzenberger involutions and Bender–Knuth involutions, given by the identity
$$s_{[1,i]}=t_1(t_2t_1)(t_3t_2t_1)\cdots (t_i\cdots t_1).$$ This relation is described in ~\cite{Berenstein_Kirillov_1996,Chmutov_Glick_Pylyavskyy_2020}, and it can be checked that the Bender--Knuth involutions generate the image of $C_n$ in the permutation group of $\syt(\lambda)$. 

\subsection{Notation}

A partition $\lambda$ of $n$ is written as $(a_1,a_2,\cdots,a_k)$ where $a_1\geq a_2\geq \cdots \geq a_k$ and $a_1+a_2+\cdots+a_k=n$. We denote by $r_i(\lambda)$ the number of boxes in row $i$ and by $c_j(\lambda)$ the number of boxes in column $j$ of the Young diagram of $\lambda$. Let $\mathcal{I}:=\{(i,j)\mid i\in [1,c_1(\lambda)], j\in [1,r_i(\lambda)]\}$ be the set of indices of $\lambda$. We view an element $A$ of $\syt(\lambda)$ as a bijection $A:\mathcal{I}\rightarrow[1,n]$ where for all $(i_1,j_1)\neq(i_2,j_2)\in \mathcal{I}$ with $i_1\leq i_2$ and $j_1\leq j_2$, we have $A(i_1,j_1)<A(i_2,j_2)$.

We call a box $(i,j)$ of $\lambda$ a \emph{corner} if there exists a standard Young tableau $A$ of this shape such that $A^{-1}(n)=(i,j)$. We additionally call $(i,j)$ an \emph{extended corner} if there exists a standard Young tableau $A$ of shape $\lambda$ such that $A^{-1}(n)=(i,j)$ and either $A^{-1}(n-1)=(i-1,j)$ or $A^{-1}(n-1)=(i,j-1)$. If $(a,b)$ is a corner of $\lambda$, we denote $\lambda\setminus (a,b)$ as $\lambda$ with the box at $(a,b)$ removed, a partition of $n-1$. We also define $A\setminus n$ to be the standard Young tableau of $\lambda\setminus A^{-1}(n)$ such that $(A\setminus n)^{-1}(k)=A^{-1}(k)$ for all $k\in[1,n-1]$.

Finally, we denote by $\lambda^T$ the transpose partition of $\lambda$, and $A^T\in\syt(\lambda^T)$ the transpose tableau of $A$.

\subsection{Single tableaux are transitive}
\begin{thm} \label{1-trans}
    Let $\lambda$ be a partition of $n$. Then, $\syt(\lambda)$ is transitive under $C_n$.
\end{thm}

\begin{proof}
    Let $A\in\syt(\lambda)$. It suffices to show there exists some $g\in C_n$ such that $g(A)=T$, where $T$ is the fixed tableau with $T(i,j)=j+r_1(\lambda)+r_2(\lambda)+\cdots +r_{i-1}(\lambda)$. Suppose $(i,j)$ be the first coordinate (ordered lexicographically) where $A(i,j)\neq T(i,j)$. Then we can always find some $h\in C_n$ such that the first coordinate that $h(A)$ and $T$ do not agree on is greater than $(i,j)$ (or $h(A)=T$). Observe that $A(i,j)>T(i,j)$ and so $t_{T(i,j)}\cdots t_{A(i,j)-2}t_{A(i,j)-1}(A)$ agrees with $T$ at the coordinate $(i,j)$, but leaves all previous coordinates untouched, so either $h(A)=T$ or the first coordinate they do not agree on must be greater than $(i,j)$).
\end{proof}

Pairs of standard Young tableaux are transitive for all shapes except hook-shaped and self-transpose partitions. In these cases, they are transitive with the exception of a few obvious invariants, which we describe in the next section.

\section{Hook-shaped tableaux and self-transpose tableaux}\label{se:hook-and-symmetric}
\subsection{Hook-shaped tableaux}

\begin{defn*}
     A partition $\lambda$ is \emph{hook-shaped} if it does not include a box at (2,2).
\end{defn*}

\begin{example} \label{ex: H, AH}
    The partition on the left is hook-shaped, whereas the partition on the right is not since it contains the box at $(2,2)$.
    \begin{center}
    $$\begin{ytableau}
        \color{white}0 &  &  & & \\
         \\
        \\
        \color{white}0
    \end{ytableau}\quad\quad\quad \quad \begin{ytableau}
        \color{white}0 &  &  & & \\
         & \\
        \\
        \color{white}0
    \end{ytableau}$$
    \end{center}
\end{example}
\vspace{0.25cm}

Throughout this section, assume $\lambda_1$ and $\lambda_2$ are hook-shaped partitions with $A\in\syt(\lambda_1)$ and $B\in\syt(\lambda_2)$.

\begin{defn*}
    We denote by $S_{A,B}=\{k \mid A^{-1}(k)=(1,a), B^{-1}(k)=(1,b) \text{ for some }a,b\}$ the set of shared values in the first rows of $A$ and $B$.
\end{defn*}

\begin{lemma} \label{invariant}
    The value $|S_{A,B}|$ is invariant under $C_n$.
\end{lemma}

\begin{proof}
    It suffices to show that $I:=|S_{A,B}|$ is invariant under any $t_k$. If $A$ and $B$ are both fixed under $t_k$, clearly $I$ is fixed. If neither $A$ nor $B$ are fixed under $t_k$, then it must be the case that $k$ and $k+1$ are swapped between the first row and the first column: If either $k$ or $k+1$ are shared in the first row, then which of $k$ or $k+1$ is shared is simply swapped in $t_k(A)$ and $t_k(B)$; if neither $k$ nor $k+1$ are shared in the first row, then $t_k$ preserves that neither are shared. If $A$ is fixed under $t_k$ and $B$ is not, then either $k$ and $k+1$ are both in the first row of $A$ or both in the first column of $A$. In the former case, $t_k$ maintains that one of $k$ or $k+1$ is shared in the first row, and in the latter case, $t_k$ maintains that neither are shared in the first row. The proof is symmetric for if $B$ is fixed and $A$ is not.
\end{proof}

\begin{lemma} \label{inc shared in hook}
    The value $m(A,B):=\max(S_{A,B})$ exists and is not equal to $n$ if and only if there exists some $g\in C_n$ such that $m(g(A),g(B))>m(A,B)$.
\end{lemma}
\begin{proof}
The reverse direction follows directly. Now, let $k=m(A,B)$. Clearly, $k+1$ is not in the first row of both $A$ and $B$. If $k+1$ is not in the first row of either, then $m(t_k(A),t_k(B))=k+1$, since $k+1$ is swapped to the first row of both. Otherwise, if $k+1$ is in the first row of either $A$ or $B$ but not the other, then $t_k$ leaves fixed the standard Young tableau with $k$ and $k+1$ in the first row and swaps $k+1$ to the first row of the other, and hence, $m(t_k(A),t_k(B))=k+1$.
\end{proof}
We are now in a position to describe the orbits of pairs standard Young tableaux of hook-shaped partitions under $C_n$.

\begin{prop} \label{2-trans hook}
    There are $\min(r_1(\lambda_1), c_1(\lambda_1), r_1(\lambda_2), c_1(\lambda_2))$ orbits of $C_n$ on the set of pairs $\syt(\lambda_1) \times \syt(\lambda_2)$.
\end{prop}

\begin{proof}
We proceed by induction. Suppose without loss of generality that $r_1(\lambda_1)\leq c_1(\lambda_1),$ $ r_1(\lambda_2),c_1(\lambda_2)$. Let $I_{r}$ denote the number of shared values in the first row of $A$ and the first row of $B$, and let $I_{c}$ denote the number of shared values in the first row of $A$ and the first column of $B$. Observe that $I_r$ and $I_c$ are both invariants by Lemma~\ref{invariant} with $I_r+I_c=r_1(\lambda_1)$. Also, there exist pairs of standard Young tableaux for which $I_r$ and $I_c$ attain values from 1 to $r_1(\lambda_1)$. It suffices to show that $I_r$ is a complete invariant (i.e., given a fixed value of $I_r$, any two pairs of standard Young tableaux in $\syt(\lambda_1) \times \syt(\lambda_2)$ with this value are in the same orbit).

Without loss of generality, we assume $I_r>1$ (otherwise, we do a symmetric proof on the row of $A$ and the column of $B$, as $I_c>1$). If $r_1(\lambda_1)=1$, then observe that $A$ is fixed under any action of $C_n$, and so by Theorem~\ref{1-trans}, there is indeed 1 orbit. Now, suppose $r_1(\lambda_1)>1$. By Lemma~\ref{inc shared in hook}, there exists $g\in C_n$ such that $(g(A))^{-1}(n)=(1,r_1(\lambda_1))$ and $(g(B))^{-1}(n)=(1,r_1(\lambda_2))$. Observe that the number of values shared in the first row of $A\setminus n$ and the first row of $B\setminus n$ is $I_r-1$, and since this is a complete invariant for the shapes $\lambda_1\setminus (1,r_1(\lambda_1))$ and $\lambda_2\setminus (1,r_1(\lambda_2))$, we have that $I_r$ is also a complete invariant.
\end{proof}

When $\lambda_1=\lambda_2$, we can deduce the sizes of the orbits purely combinatorially, since the orbits are defined by the number of shared entries in the first rows of both standard Young tableaux:

\medskip

\begin{cor}
     Hook-shaped tableaux are not 2-transitive, but the number of entries shared in the first row of the two tableaux is a complete invariant. Let $\lambda$ be a hook-shaped partition of $n$. Then the lengths of the orbits of pairs of elements of $\syt(\lambda)$ under the group $C_n$ are $N\cdot {k-1\choose 0}{l-1\choose 0},N\cdot{k-1\choose 1}{l-1\choose 1},\ldots,N\cdot{k-1\choose k-1}{l-1\choose l-k}$, where $k:=\min(r_1(\lambda),c_1(\lambda))$, $l:=\max(r_1(\lambda),c_1(\lambda))$, and $N:=|\syt(\lambda)|$.
\end{cor} 

\subsection{Self-transpose tableaux}
\begin{defn*}
    A partition $\lambda$ is \emph{self-transpose} if it is fixed under transposition (in other words, $\lambda=\lambda^T$).
\end{defn*}

\begin{prop} \label{prop: self-transpose invariant}
    Let $\lambda$ be a self-transpose partition of $n$. The set of pairs $\syt(\lambda)\times \syt(\lambda)$ has at least 3 orbits under $C_n$.
\end{prop}

\begin{proof}
    This follows from the fact that a transposed pair of tableaux stays transposed under $C_n$ (i.e., there are two trivial orbits).
\end{proof}

In the next section, we will prove that there are exactly 3 orbits; that is, we will see that transposition is the only invariant.

\section{2-transitivity}\label{se:2-trans}
The strategy to prove in general that pairs of standard Young tableaux are transitive is by induction on the number of boxes in the diagram, as removing the box containing $n$ yields a tableau on $n-1$ boxes. Since our proof does not depend on the two diagrams being the same shape we actually prove the stronger statement: given a fixed $n$, any pair of not necessarily the same partitions of $n$ is transitive under $C_n$ (excluding hook-shaped and self-transpose invariants).

First, let us be precise about which pairs of tableaux are transitive. If both partitions are hook-shaped, there is no chance that the pair of tableaux will be transitive.
\begin{defn*}
    Let $\lambda_1,\lambda_2$ be partitions of $n$. We say that the pair of partitions $(\lambda_1,\lambda_2)$ is \emph{viable} if at most one partition $\lambda_1$ or $\lambda_2$ is hook-shaped.
\end{defn*}

On the other hand, if both partitions are the same or transposed to each other, the pair of tableaux will still be transitive as long as the standard Young tableaux are not exactly the same or transposed.

\begin{defn*}
    Let $A\in\syt(\lambda_1)$ and $B\in\syt(\lambda_2)$. We say that the pair of tableaux $(A,B)$ is \emph{viable} if $(\lambda_1,\lambda_2)$ is viable and $A\neq B$ and $A^T\neq B$.
\end{defn*}

Since pairs of hook-shaped tableaux are not transitive, our base case should be those pairs where it is possible to remove a box from both and end up with hook-shaped partitions. As such, we have the following base case, where an \emph{almost hook-shaped} partition is a hook-shaped partition plus the box at $(2,2)$ (for example, the partition on the right in Example~\ref{ex: H, AH}).

\begin{defn*}
   We say that the pair $(\lambda_1,\lambda_2)$ is \emph{base case} if it is viable and either $\lambda_1$ and $\lambda_2$ are both almost hook-shaped or one of $\lambda_1$ and $\lambda_2$ is almost hook-shaped and the other is hook-shaped.
\end{defn*}

\subsection{Proof outline} Given two pairs of tableaux $(A,B)$ and $(C,D)$ (where $A$ and $C$ are of the same partition and $B$ and $D$ are of the same partition), if we can find a group element which takes $n$ to the same box in $A$ and $C$ and the same box in $B$ and $D$, in theory we can use our inductive hypothesis to find a path between these two pairs. More rigorously, we'd like to find some $g\in C_n$ such that $(g(A))^{-1}(n)=C^{-1}(n)$ and $(g(B))^{-1}(n)=D^{-1}(n)$. Our strategy for doing this is as follows: If the partition has at least 3 corners, we can find a box in which neither tableaux, say $A$ or $C$, contains $n$. Then, we can use induction to move $n-1$ to this third box in both $A$ and $C$, and consequently, applying $t_{n-1}$ to $A$ and $C$ moves $n$ into the desired box. By simultaneously performing this maneuver on $B$ and $D$, we can find a way to move $n$ into the same box in the second partition as well. (If there are fewer corners, we can do a similar trick.)

However, there's a catch. In order to use our inductive hypothesis, we also need to ensure that when the boxes containing $n$ in $(g(A),g(B))$ and $(C,D)$ are removed (as well as $(t_{n-1}g(A),t_{n-1}g(B))$ and $(t_{n-1}(C),t_{n-1}(D))$), the $(n-1)$-box partitions fall into our inductive case. Even when the two partitions are the same shape, this may not always be true.

\begin{example}
    This pair of tableaux is viable, but if we remove the boxes containing 9, the tableaux are transposed to each other and are no longer fully transitive.
    \begin{center}
    $$\begin{ytableau}
        1 & 2 & 3 & 9\\
        4 & 5 & 6\\
        7 & 8
    \end{ytableau}\quad\quad\quad \quad \begin{ytableau}
        1 & 4 & 7 & 9\\
        2 & 5 & 8\\
        3 & 6
    \end{ytableau}$$
    \end{center}
\end{example}

To ensure that we never run into this problem, we need \Cref{pre-move 2-trans}.

\subsection{Proof of 2-transitivity}

\begin{lemma} \label{pre-move 2-trans}
    Let $(\lambda_1,\lambda_2)$ be a viable pair of partitions of $n$ which is not base case, and let $A\in\syt(\lambda_1)$ and $B \in\syt(\lambda_2)$. If $(A,B)$ is viable, then there exists $g\in C_n$ such that $(g(A)\setminus n,g(B)\setminus n)$ is also viable.
\end{lemma}

\begin{proof}
\textbf{Case 1:} $\lambda_1=\lambda_2$. If $A\setminus n=B\setminus n$, then clearly $A=B$, which is not viable. Hence, we take $g=1$, as $(A\setminus n, B\setminus n)$ is already viable.

\textbf{Case 2:} $\lambda_1^T=\lambda_2$. The argument is symmetric to Case 1.

\textbf{Case 3:} $\lambda_1\neq\lambda_2$ and $\lambda_1\neq\lambda_2^T$. We can assume that either $\lambda_1$ or $\lambda_2$ must have at least 2 corners (otherwise, $n$ must always be in the only corner, and the problem is equivalent by removing the only corner). Without loss of generality, suppose $\lambda_1$ has at least two corners. By Theorem~\ref{1-trans}, there exists some $h\in C_n$ such that $(h(A))^{-1}(n-1)$ is a corner. Let $\lambda_1^*$ be the shape $\lambda_1$ with the box at $(h(A))^{-1}(n)$ removed, and let $\lambda_2^*$ be the shape $\lambda_2$ with the box at $(h(B))^{-1}(n)$ removed. If $\lambda_1^*\neq\lambda_2^*$ and $(\lambda_1^*)^T\neq\lambda_2^*$, then we are done, as clearly $h(A)\setminus n\neq h(B)\setminus n$ and $(h(A)\setminus n)^T\neq h(B)\setminus n$. If $\lambda_1^*=\lambda_2^*$ and $(\lambda_1^*)^T=\lambda_2^*$, then observe that $\lambda_1$ with the box $(t_{n-1}h(A))^{-1}(n)$ taken away cannot be the same or transposed shape as $\lambda_2$ with the box $(t_{n-1}h(B))^{-1}(n)$ taken away, so we are done. If $\lambda_1^*=\lambda_2^*$ and $(\lambda_1^*)^T\neq\lambda_2^*$, and if $h(A)\setminus n\neq h(B)\setminus n$, we are done since $(h(A)\setminus n)^T\neq h(B)\setminus n$. However, if $h(A)\setminus n = h(B)\setminus n$, then observe that $(t_{n-1}h(A)\setminus n)^T\neq t_{n-1}h(B)\setminus n$ and $t_{n-1}h(A)\setminus n\neq t_{n-1}h(B)\setminus n$ because $\lambda_1$ with the box $(t_{n-1}h(A))^{-1}(n)$ taken away is not the same shape as $\lambda_2$ with the box $(t_{n-1}h(B))^{-1}(n)$ taken away, so we are also done.
\end{proof}

\begin{thm} \label{2-trans}
    Let $(\lambda_1,\lambda_2)$ be a viable pair of partitions of $n$. The set of pairs $\{(A,B)\in \syt(\lambda_1)\times \syt(\lambda_2)\mid (A,B) \text{ is viable}\}$ is transitive under $C_n$.
\end{thm}

\begin{proof}
\textbf{Base case:} Our base case is pairs of partitions $(\lambda_1,\lambda_2)$ as specified above. To show the base case, we also use induction on the number of boxes. Let's assume that $\lambda_1$ is almost hook-shaped and $\lambda_2$ is either almost hook-shaped or hook-shaped. If we have $\lambda_1=\lambda_2=(2,2)$, or we have $\lambda_1=(3,2)$ and $\lambda_2=(3,2)$, $\lambda_2=(3,1,1)$, $\lambda_2=(2,1,1,1)$, or $\lambda_2=(1,1,1,1,1)$, or we have $\lambda_1=(3,2,1)$ and $\lambda_2=(4,2)$, $\lambda_2=(4,1,1)$, $\lambda_2=(3,2,1)$, $\lambda_2=(3,1,1,1)$, $\lambda_2=(2,2,1,1)$, $\lambda_2=(2,1,1,1,1)$, or $\lambda_2=(1,1,1,1,1,1)$, we can manually verify that the theorem holds.

Otherwise, observe that $\max(r_1(\lambda_1),c_1(\lambda_1))>3$, and suppose that the theorem holds for base case pairs with $n-1$ boxes. Let $A\in\syt(\lambda_1)$ and $B\in\syt(\lambda_2)$, and without loss of generality, suppose $r_1(\lambda_1)\geq c_1(\lambda_1)$ and $r_1(\lambda_2)\geq c_1(\lambda_2)$. It suffices to show that there exists some $g\in C_n$ such that $(g(A))^{-1}(n)=(1,r_1(\lambda_1))$ and $(g(B))^{-1}(n)=(1,r_1(\lambda_2))$. By Theorem~\ref{1-trans}, there exists some $h\in C_n$ such that $(h(A))^{-1}(n)=(1,r_1(\lambda_1))$. If $(h(B))^{-1}(n)=(1,r_1(\lambda_2))$, we are done. Otherwise, observe that since $r_1(\lambda_1)>3$, it follows that $\lambda_1$ with the box $(h(A))^{-1}(n)$ taken away is either a hook-shaped partition or an almost-hook-shaped partition with $\max(r_1(\lambda_1),c_1(\lambda_1))\geq 3$ and $\lambda_2$ with the box $(h(B))^{-1}(n)$ taken away is hook-shaped, so there exists some $h'\in C_n$ such that $(h'h(A))^{-1}(n-1)=(1,r_1(\lambda_1)-1)$ and $(h'h(B))^{-1}(n-1)=(1,r_1(\lambda_2))$. Then, $(t_{n-1}h'h(A))^{-1}(n)=(1,r_1(\lambda_1))$ and $(t_{n-1}h'h(B))^{-1}(n)=(1,r_2(\lambda_1))$, as desired.

\textbf{Inductive step:} Now, assume $(\lambda_1,\lambda_2)$ is not base case, and let $A\in\syt(\lambda_1)$ and $B\in\syt(\lambda_2)$. We know the theorem holds for all pairs of partitions of $n-1$. We now show that for any $C\in\syt(\lambda_1)$ and $D\in\syt(\lambda_2)$ where $A\neq B,B^T$ and $C\neq D,D^T$, there exists some $g\in C_n$ such that $g(A)=C,g(B)=D$.

By \Cref{pre-move 2-trans}, we can find $g_1,g_2\in C_n$ such that the pairs $(g_1(A)\setminus n, g_1(B)\setminus n$ and $(g_2(C)\setminus n, g_2(D)\setminus n)$ are both viable. We now select a position to move $n-1$ to in $\lambda_1$ for $g_1(A)$ and $g_2(C)$.

If $\lambda_1$ has more than two corners, we choose the corner $(i,j)$ for both $g_1(A)$ and $g_2(C)$ such that $(i,j)\neq (g_1(A))^{-1}(n)$ and $(i,j)\neq (g_2(C))^{-1}(n)$. If $\lambda_1$ has two corners or fewer, then observe that $\lambda_1$ has an extended corner $(k,l)$. Without loss of generality, suppose $(k,l)=(g_1A)^{-1}(n)$. For $g_1(A)$, we select $(k,l-1)$ or $(k-1,l)$ (whichever it is possible to place $n-1$ in by the definition of an extended corner), and for $g_2(C)$, we select $(k,l)$.

We do the same process to select a position to move $n-1$ to in $\lambda_2$ for $g_1(B)$ and $g_2(D)$.
Since $(g_1(A)\setminus n,g_1(B)\setminus n)$ and $(g_2(C)\setminus n,g_2(D)\setminus n)$ are viable pairs on $n-1$ boxes, we can use our inductive hypothesis to find $h_1,h_2\in C_n$ such that $(h_1g_1(A))^{-1}(n-1),(h_1g_1(B))^{-1}(n-1),(h_2g_2(C))^{-1}(n-1),$ and $(h_2g_2(D))^{-1}(n-1)$ are in the specified positions as given above (so that after applying $t_{n-1}$ to both pairs, the position of $n$ is the same in the $A$ and $C$ as well as the $B$ and $D$ tableaux). Since we want to apply $t_{n-1}$ to both pairs and then use our inductive hypothesis yet again, we also need to ensure the pairs restricted to $n-1$ boxes will still be viable after applying $t_{n-1}$. Since a viable pair remains viable even after adding a box to any position, this is true as long as the $(n-2)$-box subdiagrams of both pairs are viable. Since all $(n-2)$-box subdiagrams have at least 3 standard Young tableaux, we can require $h_1$ and $h_2$ to also keep both $(n-2)$-box subdiagrams viable. Hence, our choice of $h_1$ and $h_2$ is such that $(h_1g_1(A)\setminus (n,n-1),(h_1g_1(B)\setminus (n,n-1))$ and $(h_2g_2(C)\setminus (n,n-1),(h_2g_2(D)\setminus (n,n-1))$ are viable; and $(t_{n-1}h_1g_1(A))^{-1}(n)=(t_{n-1}h_2g_2(C))^{-1}(n)$ and $(t_{n-1}h_1g_1(B))^{-1}(n)=(t_{n-1}h_2g_2(D))^{-1}(n)$). 

Thus, we have that the pairs $(t_{n-1}h_1g_1(A)\setminus n, t_{n-1}h_1g_1(B)\setminus n)$ and $(t_{n-1}h_2g_2(C)\setminus n, t_{n-1}h_2g_2(D)\setminus n)$ are both viable. Hence, we can use our inductive hypothesis once again to find $f_1,f_2\in C_n$ such that $f_1t_{n-1}h_1g_1(A)=f_2t_{n-1}h_2g_2(C)$ and $f_1t_{n-1}h_1g_1(B)=f_2t_{n-1}h_2g_2(D)$, so the group element $(f_2t_{n-1}h_2g_2)^{-1}f_1t_{n-1}h_1g_1$ transforms the pair $(A,B)$ into $(C,D)$, as desired.
\end{proof}

\begin{cor}\label{co:2-trans}
     Let $\lambda$ be a non hook-shaped partition of $n$. If $\lambda$ is self-transpose, the set of pairs $\syt(\lambda)\times \syt(\lambda)$ has 3 orbits under $C_n$; otherwise, there are 2 orbits.
\end{cor}

\begin{proof}
    This follows from \Cref{prop: self-transpose invariant}.
\end{proof}

\section{3-transitivity}\label{se:3-trans}
The proof structure is the same as for 2-transitivity, with a little more casework. We also need to update our inductive case and base case as follows:

\begin{defn*}
    Let $\lambda_1,\lambda_2,\lambda_3$ be partitions of $n$. We say that the triple of partitions $(\lambda_1,\lambda_2,\lambda_3)$ is \emph{viable} if at most one partition $\lambda_1,\lambda_2,\lambda_3$ is hook-shaped.
\end{defn*}

\begin{defn*}
    Let $A\in\syt(\lambda_1),B\in\syt(\lambda_2),C\in\syt(\lambda_3)$. We say that the triple of tableaux $(A,B,C)$ is \emph{viable} if $(\lambda_1,\lambda_2,\lambda_3)$ is viable and the following conditions are satisfied:
    \begin{enumerate}
        \item No two tableaux $A,B,C$ are the same.
        \item No two tableaux $A,B,C$ are transpositions of each other.
    \end{enumerate}
\end{defn*}

\begin{defn*}
    We say that the triple $(\lambda_1,\lambda_2,\lambda_3)$ is \emph{base case} if it is viable and satisfies one of the following conditions:
    \begin{enumerate}
        \item Two partitions of $\lambda_1,\lambda_2,\lambda_3$ are almost hook-shaped.
        \item One partition of $\lambda_1,\lambda_2,\lambda_3$ is almost hook-shaped and one is hook-shaped.
    \end{enumerate}
\end{defn*}

\begin{lemma} \label{pre-move}
    Let $(\lambda_1,\lambda_2,\lambda_3)$ be a viable triple of partitions of $n$ (each with at least two corners) that are not base case, and let $A\in\syt(\lambda_1)$, $B \in\syt(\lambda_2)$, and $C\in\syt(\lambda_3)$. If $(A,B,C)$ is viable, then there exists $g\in C_n$ such that $(gA\setminus n,gB\setminus n,gC\setminus n)$ is also viable.
\end{lemma}

\begin{proof}
\textbf{Case 1:} $\lambda_1=\lambda_2$ or $\lambda_1=\lambda_2^T$. Without loss of generality, suppose $\lambda_1=\lambda_2$. Using \Cref{2-trans}, let $g\in C_n$ such that $n-1$ is in a corner of $g(C)$; $g(A)\setminus n$ and $g(C)\setminus n$ are not the same as or transposed to each other; and $t_{n-1}g(A)\setminus n$ and $t_{n-1}g(C)\setminus n$ are not the same as or transposed to each other. If $g(A)\setminus n=g(B)\setminus n$ (resp., $g(A)\setminus n=(g(B)\setminus n)^T$), then $A=B$ (resp., $A=B^T$) since $\lambda_1=\lambda_2$, which contradicts that $(A,B,C)$ is viable. So, $g(A)\setminus n$ cannot be the same as or transposed to $g(B)\setminus n$. Finally, if $g(B)\setminus n\neq g(C)\setminus n$ and $g(B)\setminus n\neq (g(C)\setminus n)^T$, we are done. Otherwise, since $n-1$ is in a corner of $g(C)$, we have that $t_{n-1}g(B)\setminus n\neq t_{n-1}g(C)\setminus n$ and $t_{n-1}g(B)\setminus n\neq (t_{n-1}g(C)\setminus n)^T$, and we are also done.

\textbf{Case 2:} No two partitions of $\lambda_1,\lambda_2,\lambda_3$ are the same as or transposed to each other. There exists some box $(k,l)$ such that $\lambda_1\setminus (k,l)$ is a subshape of $\lambda_2$, and there exists some box $(k',l')$ such that $\lambda_1\setminus (k',l')$ is a subshape of $\lambda_2^T$. If $(k,l)$ and $(k',l')$ are the only corners of $\lambda_1$, then $\lambda_1=\lambda_2$ or $\lambda_1=\lambda_2^T$, and we may refer to Case 1.

Otherwise, there is some corner $(i,j)\in\lambda_1$ such that $\lambda_1\setminus(i,j)\notin \lambda_2$ and $\lambda_1\setminus(i,j)\not\subset \lambda_2^T$. Using \Cref{2-trans}, let $g\in C_n$ such that $(g(A))^{-1}(n)=(i,j)$, $n-1$ is in a corner of $g(C)$, $g(A)\setminus n$ and $g(C)\setminus n$ are not the same as or transposed to each other, and $t_{n-1}g(A)\setminus n$ and $t_{n-1}g(C)\setminus n$ are not the same as or transposed to each other. Since the shape of $g(A)\setminus n$ contains a box which is not in $\lambda_2$, it cannot be the case that $g(A)\setminus n$ is the same as or transposed to $g(B)\setminus n$. If $g(B)\setminus n\neq g(C)\setminus n$ and $g(B)\setminus n\neq (g(C)\setminus n)^T$, we are done. Otherwise, since $n-1$ is in a corner of $g(C)$, we have that $t_{n-1}g(B)\setminus n\neq t_{n-1}g(C)\setminus n$ and $t_{n-1}g(B)\setminus n\neq (t_{n-1}g(C)\setminus n)^T$, and we are also done.
\end{proof}

\begin{thm}\label{th:3-trans}
    Let $\lambda_1,\lambda_2,\lambda_3$ be partitions of $n$. Then the set of viable triples $A\in\syt(\lambda_1),B\in\syt(\lambda_2),C\in\syt(\lambda_3)$ is transitive under $C_n$.
\end{thm}

\begin{proof}
\textbf{Base Case:} We also use induction to show the base case. Let's assume that $\lambda_1$ is almost hook-shaped, $\lambda_2$ is either almost hook-shaped or hook-shaped, and $\lambda_3$ is any partition (which is not hook-shaped). (Throughout the proof, we  assume that partitions have at least two corners, as otherwise, the position of $n$ in is fixed in the only corner, and we are done by \Cref{2-trans}.)

\begin{table}
\centering
\begin{tabular}{l|l|l|l}
 & $\lambda_1$ & $\lambda_2$ & $\lambda_3$  \\\hline
 $n=4$ & $(2,2)$                          & all                           & all            \\\hline
 $n=5$ & $(3,2)$                           & all                         & all                \\
 & $(2,2,1)$                         & all                          & all          \\\hline
 $n=6$ & $(3,2,1)$                         & all                          & all           \\
 & $(4,2)$                          & all                       & all        \\
 & $(2,2,1,1)$                      & all                 &  all     
\end{tabular}
\caption{The 3-transitive cases we check by hand. Note: ``all" refers to all hook-shaped and almost hook-shaped partitions (of $n$ boxes) in the $\lambda_2$ column and all non-hook-shaped partitions in the $\lambda_3$ column. \label{table}}
\end{table}

We manually verify that the theorem holds for the cases in \Cref{table}. Otherwise, observe that $\max(r_1(\lambda_1),c_1(\lambda_1))>3$ and $\max(r_1(\lambda_2),c_1(\lambda_2))>3$, and suppose that the theorem holds for base case triples of partitions of $n-1$. Let $A\in\syt(\lambda_1),B\in\syt(\lambda_2),C\in\syt(\lambda_3)$, and without loss of generality, suppose the first column is not longer than the first row in $\lambda_1$ and $\lambda_2$ (that is, $r_1(\lambda_1)\geq c_1(\lambda_1)$ and $r_1(\lambda_2)\geq c_1(\lambda_2)$). Fix $(k,l)$ to be any corner of $\lambda_3$. We now show there is a way to move $n$ into $(k,l)$ of $\lambda_3$ and the first row of $\lambda_1$ and $\lambda_2$.

Using \Cref{2-trans}, let $g\in C_n$ such that $(g(A))^{-1}(n)=(1,r_1(\lambda_1))$ and $(g(C))^{-1}(n)\neq (k,l)$. Observe that $(g(A)\setminus n,g(B)\setminus n,g(C)\setminus n)$ is base case. Since $r_1(\lambda_1)>3$, we use induction to find some $h\in C_n$ such that $(hg(A))^{-1}(n-1)=(1,r_1(\lambda_1)-1)$, $(hg(C))^{-1}(n-1)=(k,l)$, and lastly, $(hg(B))^{-1}(n-1)=(1,r_1(\lambda_2)-1)$ if $(g(B))^{-1}(n)=(1,r_1(\lambda_2)$ or otherwise $(hg(B))^{-1}(n-1)=(1,r_1(\lambda_2))$ (while maintaining the positions of $n$ in all three tableaux). Because $n-1$ and $n$ are adjacent in $hg(A)$ and $r_1(\lambda_2)>2$, the box $(k,l)$ of $t_{n-1}hg(C)$ and the first rows of $t_{n-1}hgA$ and $t_{n-1}hg(B)$ contain $n$, as desired.

\textbf{Inductive step:} Now, assume the theorem holds for partitions of $n-1$. Let $A,D\in\syt(\lambda_1)$, $B,E \in\syt(\lambda_2)$, and $C,F\in\syt(\lambda_3)$ such that $(A,B,C)$ and $(D,E,F)$ be viable triples. Use \Cref{pre-move} to find $g_1,g_2\in C_n$ such that $(g_1(A)\setminus n,g_1(B)\setminus n,g_1(C)\setminus n)$ and $(g_2(D)\setminus n,g_2(E)\setminus n,g_2(F)\setminus n)$ are viable.

Our goal is to move $n$ into the same position in $g_1(A)$ and $g_2(D)$ (resp., $g_1(B)$ and $g_2(E)$, $g_1(C)$ and $g_2(F)$). We do this by selecting a location to move $n-1$ to in both tableaux of each partition. Recall that we can assume that $\lambda_1,\lambda_2,\lambda_3$ all have at least two corners.

Let us consider $\lambda_1$ (the process is the same for $\lambda_2$ and $\lambda_3)$. If $\lambda_1$ has more than two corners, let $(i_A,j_A)$ and $(i_D,j_D)$ be the same corner which is not equal to $g_1(A)^{-1}(n)$ or $g_2(D)^{-1}(n)$. Otherwise, if $\lambda_1$ has exactly two corners, one such corner must be an extended corner. Without loss of generality, suppose $(k,l):=g_1(A)^{-1}(n)$ is the extended corner and $g_2(D)^{-1}(n)$ is the other corner. Then, let $(i_A,j_A):=(k-1,l)$ or $(i_A,j_A):=(k,l-1)$ (whichever box is possible to place $n-1$ in), and let $(i_D,j_D):=(k,l)$.

Now, repeat this process for $\lambda_2$ and $\lambda_3$ to find positions $(i_B,j_B),(i_E,j_E)$ and $(i_C,j_C),(i_F,j_F)$ to move $n-1$ into. By induction, there exist $h_1,h_2\in C_n$ such that $n-1$ is in the specified positions in the triples $(h_1g_1(A),h_1g_1(B),h_1g_1(C))$ and $(h_2g_2(D),h_2g_2(E),h_2g_2(F))$ such that when they are restricted to $n-2$ boxes, both triples are still viable. Finally, we have that $n$ is in the same position in $t_{n-1}h_1g_1(A)$ and $t_{n-1}h_2g_2(D)$ (resp., $t_{n-1}h_1g_1(B)$ and $t_{n-1}h_2g_2(E)$, $t_{n-1}h_1g_1(C)$ and $t_{n-1}h_2g_2(F)$). Using the same argument as in \Cref{2-trans}, both triples $(t_{n-1}h_1g_1(A)\setminus n, t_{n-1}h_1g_1(B)\setminus n,t_{n-1}h_1g_1(C)\setminus n)$ and $(t_{n-1}h_2g_2(D)\setminus n, t_{n-1}h_2g_2(E)\setminus n,t_{n-1}h_2g_2(F)\setminus n)$ are viable. Hence, we can find $f_1,f_2\in C_n$ such that $f_1t_{n-1}h_1g_1(A)=f_2t_{n-1}h_2g_2(C)$, $f_1t_{n-1}h_1g_1(B)=f_2t_{n-1}h_2g_2(D)$, and $f_1t_{n-1}h_1g_1(C)=f_2t_{n-1}h_2g_2(F)$. Thus, the group element $(f_2t_{n-1}h_2g_2)^{-1}f_1t_{n-1}h_1g_1$ transforms the triple $(A,B,C)$ into $(D,E,F)$, as desired.
\end{proof}

As a consequence, we get the following (Theorem~B in the Introduction):
\begin{cor}\label{co:3-trans}
    Suppose that $\lambda$ is not hook-shaped and not self-transpose. Let $N$ be the cardinality of $\syt(\lambda)$. Then the action of $C_n$ on the set $\syt(\lambda)$ is $3$-transitive.
\end{cor}

\section{Proof of Theorem B}\label{se:thm-B}

In this section, we deduce the following statement (Theorem~B in the Introduction) from Corollary~\ref{co:3-trans}:
\begin{thm}\label{th:B}
    Suppose that $\lambda$ is not hook-shaped and not self-transpose. Let $N$ be the cardinality of $\syt(\lambda)$. Then the image of $C_n$ in the permutation group $S_N$ of the set $\syt(\lambda)$ is either the whole $S_N$ or the alternating group $A_N$.
\end{thm}

\begin{proof}

$3$-transitive actions apart from the standard actions of $S_N$ and $A_N$ are classified as follows, see e.g. \cite[Chapter 7]{Cameron_book}:


\medskip

\begin{tabular}{|l|l|c|}
\hline
\textbf{Group} & \textbf{Set Acted Upon} & \textbf{Degree of Transitivity} \\
\hline
$M_{11}$ & 11 points & 4-transitive \\
$M_{11}$ & 12 points & 3-transitive \\
$M_{12}$ & 12 points & 5-transitive \\
$M_{22}$ or $M_{22}.2$ & 22 points & 3-transitive \\
$M_{23}$ & 23 points & 4-transitive \\
$M_{24}$ & 24 points & 5-transitive \\
\hline
$\mathrm{PSL}(2,q)\subset G\subset\mathrm{P\Gamma L}(2,q)$, $q\ge5$ & $\mathbb{P}^1(\mathbb{F}_q)$ (q+1 points) & 3-transitive \\
\hline
$\mathrm{AGL}(d, 2)$, $d\ge3$ & $\mathbb{F}_{2}^{d}$  ($2^d$ points) & 3-transitive \\
$2^4.A_7\subset\mathrm{AGL}(4, 2)$ & $\mathbb{F}_2^4$ ($16$ points) & 3-transitive \\
\hline
\end{tabular}

\medskip

Here, the first six are standard actions of the Mathieu groups, the next is a group of projective transformations of a line over a finite field $\mathbb{F}_q$ possibly twisted with automorphisms of $\mathbb{F}_q$, and the last two examples are (subgroups of) affine transformations of a space over the $2$-element field $\mathbb{F}_2$. The general idea of proof of the Theorem is to show that the image of the cactus group $C_n$ in the group of permutations of $\syt(\lambda)$ has too many involutions with pairwise different numbers of fixed points, which excludes all of the above cases. 

\subsection{Involutions with many fixed points}

\begin{prop}\label{pr:fixed-points-of-involutions}
\begin{enumerate}
    \item For any non-hook-shaped $\lambda\vdash n\ge5$, the involutions $t_3,t_2,t_2t_4$ are non-trivial on $\syt(\lambda)$.
    \item For any non-hook-shaped $\lambda$ that does not fit into the rectangle $3\times 3$, the involutions $t_3,t_2,t_2t_4$ have pairwise different numbers of fixed points in $\syt(\lambda)$.
    \item For any non-hook-shaped $\lambda\vdash n\ge7$, the number of fixed points for each of these involutions is strictly bigger than $\frac{|\syt(\lambda)|}{8}$. 
    \item For any non-hook-shaped $\lambda\vdash n\ge6$, the number of $t_3$-fixed points on $\syt(\lambda)$ is bigger than $\frac{|\syt(\lambda)|}{3}$. 
\end{enumerate}
\end{prop}

\begin{proof}
    The first assertion is seen from the following table showing the numbers of fixed points of $t_3$, $t_2$, and $t_2 t_4$ for all partitions of 5:

\begin{tabular}{|c|c|c|c|c|}
\hline
\textbf{Partition of 5} & \#SYT & $t_3$ & $t_2$ & $t_2 t_4$ \\
\hline
(5) & 1 & 1 & 1 & 1 \\
(4,1) & 4 & 2 & 2 & 0 \\
(3,2) & 5 & 3 & 1 & 1 \\
(3,1,1) & 6 & 2 & 2 & 2 \\
(2,2,1) & 5 & 3 & 1 & 1 \\
(2,1,1,1) & 4 & 2 & 2 & 0 \\
(1,1,1,1,1) & 1 & 1 & 1 & 1 \\
\hline
\end{tabular}

    Next, for a given $\lambda\vdash n$, the numbers $|\syt(\lambda)|$ and the numbers of fixed points for the above involutions in $\syt(\lambda)$ are the sums of those numbers over all the Young diagrams obtained from $\lambda$ by removing a corner box. So once the third and the fourth assertions are true for $n=6$ and $n=7$, they hold true for bigger $n$ as well. This is see from the following tables displaying the numbers of fixed points of $t_3$, $t_2$, and $t_2 t_4$ for all partitions of 6 and 7, respectively:

\begin{tabular}{|c|c|c|c|c|}
\hline
\textbf{Partition of 6} & \textbf{\#SYT} & $t_3$ & $t_2$ & $t_2 t_4$ \\
\hline
(6) & 1 & 1 & 1 & 1 \\
(5,1) & 5 & 3 & 3 & 1 \\
(4,2) & 9 & 5 & 3 & 1 \\
(4,1,1) & 10 & 4 & 4 & 2 \\
(3,3) & 5 & 3 & 1 & 1 \\
(3,2,1) & 16 & 8 & 4 & 4 \\
(3,1,1,1) & 10 & 4 & 4 & 2 \\
(2,2,2) & 5 & 3 & 1 & 1 \\
(2,2,1,1) & 9 & 5 & 3 & 1 \\
(2,1,1,1,1) & 5 & 3 & 3 & 1 \\
(1,1,1,1,1,1) & 1 & 1 & 1 & 1 \\
\hline
\end{tabular}

\begin{tabular}{|c|c|c|c|c|}
\hline
\textbf{Partition of 7} & \#SYT & $t_3$ & $t_2$ & $t_2 t_4$ \\
\hline
(7) & 1 & 1 & 1 & 1 \\
(6,1) & 6 & 4 & 4 & 2 \\
(5,2) & 14 & 8 & 6 & 2 \\
(4,3) & 14 & 8 & 4 & 2 \\
(5,1,1) & 15 & 7 & 7 & 3 \\
(4,2,1) & 35 & 19 & 11 & 7 \\
(3,3,1) & 21 & 11 & 5 & 5 \\
(3,2,2) & 21 & 11 & 5 & 5 \\
(4,1,1,1) & 20 & 8 & 8 & 4 \\
(3,2,1,1) & 35 & 19 & 11 & 7 \\
(2,2,2,1) & 14 & 8 & 4 & 2 \\
(3,1,1,1,1) & 15 & 7 & 7 & 3 \\
(2,2,1,1,1) & 14 & 8 & 6 & 2 \\
(2,1,1,1,1,1) & 6 & 4 & 4 & 2 \\
(1,1,1,1,1,1,1) & 1 & 1 & 1 & 1 \\
\hline
\end{tabular}

    Finally, by the same reason, for any $\lambda$, the numbers of fixed points for $t_3$, $t_2$ and $t_2t_4$ form a non-increasing sequence. Moreover, it is seen from the above tables that for any non-hook-shaped Young diagram that does not fit into the rectangle $3\times3$, these numbers strictly decrease. This proves the second assertion of the Proposition.
\end{proof}

\subsection{Excluding Mathieu groups}

\begin{prop}
    Let $N=|\syt(\lambda)|$ The image of $C_n$ in the permutation group $S_N$ of $\syt(\lambda)$ is never the Mathieu subgroup $M_N\subset S_N$ and it also cannot be $M_{22}.2\subset S_{22}$.
\end{prop}

\begin{proof}
By the second assertion of Proposition~\ref{pr:fixed-points-of-involutions}, the involutions $t_3,t_2,t_2t_4$ have pairwise different and nonzero numbers of fixed points in $\syt(\lambda)$, so their images in $S_N$ represent pairwise different conjugacy classes. On the other hand, according to \cite{Conway1985ATLAS}, Mathieu groups are known to have at most $2$ conjugacy classes of non-trivial involutions. The group $M_{22}.2$ has $3$ conjugacy classes of involutions, but the elements of one of them act on the $22$-element set without fixed points, so it is also impossible.
\end{proof}

\subsection{Excluding subgroups of $\rm{P}\Gamma\rm{L}(2,q)$}

\begin{prop}
    It is impossible that $\syt(\lambda)=\mathbb{P}^1(\mathbb{F}_q)$ with $q\ge5$ and the image of $C_n$ is in $\rm{P}\Gamma\rm{L}(2,q)$.
\end{prop}

\begin{proof}
    
    Note that under the above assumptions, $n\ge5$ so $t_3$ has at least $3$ fixed points on $\syt(\lambda)$. Suppose that $t_3$ is represented by an element $g\circ\tau\in \rm{P}\Gamma\rm{L}(2,q)$ where $g\in PGL(2,q)$ and $\tau$ is an automorphism of $\mathbb{F}_q$. Then, up to conjugation with $PGL(2,q)$, we can assume that this element fixes the points $0,1,\infty\in\mathbb{P}^1(\mathbb{F}_q)$, so, up to conjugation with $PGL(2,q)$, it is just the field automorphism $\tau$. The fixed points of $\tau$ are $\mathbb{P}^1(\mathbb{F}_q^\tau)$, the projective line over the field of invariants. $\mathbb{F}_q^\tau$ is a subfield in $\mathbb{F}_q$ so since $\tau$ is an involution it is some $\mathbb{F}_l$ with $q=l^2$. In particular, the ratio of the cardinalities of $\mathbb{P}^1(\mathbb{F}_q)$ and $\mathbb{P}^1(\mathbb{F}_q^\tau)$ is $\frac{l^2+1}{l+1}$ which is greater than $3$ unless $q=9, l=3$ (in the latter case $|\syt(\lambda)|=10$ which never happens for non-hook-shaped $\lambda$). By the last assertion of Proposition~\ref{pr:fixed-points-of-involutions}, this ratio should be greater than $3$, so we have a contradiction.
\end{proof}

\subsection{Excluding subgroups of $AGL(d,2)$}
It remains to show that it can never happen that $\syt(\lambda)=\mathbb{F}_2^d$ and the image of $C_n$ is $AGL(d,2)$. We will do this by comparing the numbers of fixed points of involutions in $AGL(d,2)$ and of those in the image of $C_n$.

\begin{prop}
    Suppose that $\lambda\vdash n$ is a non-hook-shaped, non-self-transpose Young diagram such that $|\syt(\lambda)|=2^d$ for some $d$ (so there is a bijection $\syt(\lambda)\cong\mathbb{F}_2^d$). Then the image of $C_n$ is not contained in $AGL(d,2)$.
\end{prop}

\begin{proof}
    According to the third statement of Proposition~\ref{pr:fixed-points-of-involutions}, under the above assumptions, the involutions $t_3,t_2,t_2t_4$ have pairwise different numbers of fixed points in $\syt(\lambda)$, all strictly bigger than $2^{d-3}$. On the other hand, fixed points of any affine transformation of $\mathbb{F}_2^d$ form an affine subspace in $\mathbb{F}_2^d$, so their number can only be a power of $2$, so it can be only $2^{d-1}$ or $2^{d-2}$. This is a contradiction.
\end{proof}

This completes the proof of the Theorem.  
\end{proof}

\section{Galois groups of Bethe ansatz equations and of Schubert problems}\label{se:sottile}

\subsection{Bethe ansatz in Gaudin model}
In \cite{Halacheva_Kamnitzer_Rybnikov_Weekes_2020}, Halacheva, Kamnitzer, Weekes, and the second author interpreted the action $C_n$-action on a Kashiwara crystal as a \emph{monodromy action} on solutions to Bethe ansatz in the Gaudin model \cite{Gaudin2014} attached to a semisimple Lie algebra $\mathfrak{g}$ (that is a completely integrable quantum spin chain arising as a degeneration of the XXX Heisenberg model). Algebraically, the Gaudin model is the problem of simultaneous diagonalization of certain pairwise commuting operators $H_i(\underline{z})$ called \emph{Gaudin Hamiltonians} on the space of states being the Hom-space
\[
\text{Hom}_{\mathfrak{g}}(V(\mu),V(\lambda^{(1)})\otimes\ldots\otimes V(\lambda^{(n)})).
\]

Here $V(\lambda)$ is the irreducible finite-dimensional representation of a semisimple Lie algebra $\mathfrak{g}$ with the highest weight $\lambda$. The commuting operators $H_i(\underline{z})$ in question depend on a collection \( \underline{z}=(z_1,\ldots,z_n) \) of pairwise different complex numbers (see e.g. \cite{Frenkel2005} for the details). The joint eigenvalues are uniquely determined by solutions of \emph{Bethe ansatz equations}, which are algebraic equations having the $z_i$'s as parameters, see e.g. \cite{Feigin1994Gaudin}. The Galois group of these Bethe ansatz equations (more precisely, the Galois group of the extension of the field $\mathbb{C}(z_1,\ldots,z_n)$ by the eigenvalues of the Gaudin Hamiltonians) naturally permutes the joint eigenvalues of the $H_i(\underline{z})$.

The parameter $\underline{z}$ can be regarded as an element of the moduli space of stable rational curves with $n+1$ marked points \( M_{0,n+1} \), and the commuting operators $H_i(\underline{z})$ can also be extended to the Deligne-Mumford compactification \( \overline{M}_{0,n+1} \) of this moduli space (see \cite{AguirreFelderVeselov2011} for the details). According to \cite{Rybnikov2018}, the same holds for maximal commutative subalgebras of operators in the space of states containing Hamiltonians $H_i(\underline{z})$. Moreover, according to \cite{Rybnikov2020,Halacheva_Kamnitzer_Rybnikov_Weekes_2020} for the real values of the parameter, i.e. for \( \underline{z}\in\overline{M}_{0,n+1}(\mathbb{R}) \) their joint eigenvalues have no multiplicities. So the cactus group $C_n=\pi_1^{S_n}(\overline{M}_{0,n+1}(\mathbb{R}))$ acts naturally by transporting the joint eigenvalues of the Gaudin Hamiltonians in the space of states along certain paths in \( \overline{M}_{0,n+1}(\mathbb{R}) \). All such monodromy transformations of the set of joint eigenvalues of the Gaudin Hamiltonians come from some elements of the Galois group of the Gaudin eigenproblem, so we have a homomorphism from the cactus group $C_n$ to that Galois group. In general, this homomorphism is far from being subjective, so the image of $C_n$ can be regarded as a relatively small but understandable piece of the whole Galois group. Namely, the main result of \cite{Halacheva_Kamnitzer_Rybnikov_Weekes_2020} establishes the isomorphism of this $C_n$-action with the one on the corresponding multiplicity space for Kashiwara crystals, 
\[
\text{Hom}(B(\mu),B(\lambda^{(1)})\otimes\ldots\otimes B(\lambda^{(n)})),
\]
with the action of $C_n$ coming from the coboundary category formalism of \cite{HK}.

For $\mathfrak{g}=\mathfrak{gl}_d$, the highest weights $\lambda^{(i)}$ are just partitions (or Young diagrams) of the height not greater than $d$. Moreover, by the combinatorial version of Schur-Weyl duality (called \emph{Robinson-Schensted correspondence}), in the special case when $\lambda^{(i)}=(1)$ (i.e. is a $1$-box Young diagram) and $\mu$ is an arbitrary partition of $n$, the latter $C_n$-set is $\syt(\mu)$ with the $C_n$-action given by partial Sch\"utzenberger involutions. So, by Theorem~\ref{th:B} if $\mu$ non-hook-shaped and not self-transpose partition of $n$, the image of $C_n$ is large enough to cover (almost) the whole Galois group, and we have the folowing:

\begin{cor}\label{co:Bethe-ansatz} (Theorem~C in the Introduction) Let the weights $\mu,\lambda^{(1)},\ldots,\lambda^{(n)}$ be as follows \begin{itemize}
    \item $\mu$ non-hook-shaped and not self-transpose partition of $n$;
    \item  for $i>0$, $\lambda^{(i)}=(1)$ i.e. one-box Young diagrams.
\end{itemize} 
The Galois group of solutions to Bethe ansatz equations for the Gaudin model corresponding to $\mu,\lambda^{(1)},\ldots,\lambda^{(n)}$ is at least the alternating group.
\end{cor}

\subsection{Schubert problem in Grassmannians} A \emph{Schubert problem on a Grassmannian}  $\mathrm{Gr}(d, N)$ asks for the number of $d$-dimensional subspaces in a $N$-dimensional vector space satisfying incidence conditions with respect to general flags. More formally, Let $V$ be an $N$-dimensional vector space and let $F_\bullet$ be a complete flag:
\[
0 = F_0 \subset F_1 \subset F_2 \subset \cdots \subset F_N = V, \quad \dim F_i = i.
\]

Let $\lambda = (\lambda_1, \lambda_2, \ldots, \lambda_d)$ be a partition with $n-d \geq \lambda_1 \geq \lambda_2 \geq \cdots \geq \lambda_d \geq 0$.

The \emph{Schubert cell} $\Omega_\lambda^\circ(F_\bullet)$ in the Grassmannian $\mathrm{Gr}(d, N)$ is defined as:
\[
\Omega_\lambda^\circ(F_\bullet) = \left\{ W \in \mathrm{Gr}(d, V) \;\middle|\; \dim\left( W \cap F_{N - d + i - \lambda_i} \right) = i \text{ for } 1 \leq i \leq d \right\}
\]
where $W$ is a $d$-dimensional subspace of $V$.

The closure $\Omega_\lambda(F_\bullet)$ (the Schubert variety) is given by:
\[
\Omega_\lambda(F_\bullet) = \left\{ W \in \mathrm{Gr}(d, V) \;\middle|\; \dim\left( W \cap F_{N - d + i - \lambda_i} \right) \geq i \text{ for } 1 \leq i \leq d \right\}.
\]

The dimension of the Schubert cell $\Omega_\lambda^\circ$ in the Grassmannian $\mathrm{Gr}(d, N)$ associated to the partition $\lambda = (\lambda_1, \ldots, \lambda_d)$ is
\[
\dim \Omega_\lambda^\circ = d(N - d) - |\lambda| = d(N - d) - \sum_{i=1}^d \lambda_i,
\]
where $|\lambda| = \sum_{i=1}^d \lambda_i$ is the number of boxes in the Young diagram corresponding to $\lambda$. That is, the codimension of $\Omega_\lambda^\circ$ in $\mathrm{Gr}(d, N)$ is just $|\lambda|$.

Next, one considers the intersection of Schubert varieties with respect to several different flags \(F_\bullet^{(j)}\) corresponding to partitions $\lambda^{(j)}$. A Schubert problem involves choosing conditions so that this intersection is zero-dimensional and transverse, i.e. the total number of boxes of the Young diagrams $\sum |\lambda^{(j)}|$ equals $d(N-d)$ and the flags \(F_\bullet^{(j)}\) are in generic position, so the intersection is a finite number of points.

The monodromy group of a Schubert problem arises by considering how the solutions vary as the defining flags change continuously. Algebraically, one studies the function field of the parameter space of flags and the \emph{Galois group} of the splitting field of the solutions over it. In \cite{SW}, Sottile and White study such Galois groups. Namely, they prove that such Galois groups are doubly transitive for all Schubert problems involving only special Schubert conditions, as well as for all problems on Grassmannians of 2-planes (\(\mathrm{Gr}(2, n)\)) and 3-planes (\(\mathrm{Gr}(3, n)\)). Moreover, in the latter two cases, the Galois group is at least the alternating group. Our Theorem~C implies the following result, extending this to many new cases (Theorem~D in the introduction). Namely, let $\mu\vdash n$ be a non-hook-shaped and not self-transpose, of the height not greater than $d$ and let $\lambda^{(0)}$ be the Young diagram obtained as the complement of $\mu$ in the $d\times (2n-d)$ rectangle. 

\begin{cor}
    Consider the Schubert problem in the Grassmannian $\textrm{Gr}(d,2n)$ of $d$-dimensional planes in the $2n$-dimensional vector space for the collection of partitions $\lambda^{(0)},\lambda^{(1)},\ldots,\lambda^{(n)}$ with $\lambda^{(0)}$ being as above and all other $\lambda^{(i)}=(1)$, i.e. one-box Young diagrams. The Galois group of this Schubert problem is at least the alternating group.    
\end{cor}

\begin{proof}
    In \cite{MTV_Schubert,MTV_symmetric_gp}, it is shown that the Gaudin eigenproblem is the particular case of the Schubert problem. Namely, the the commutative subalgebra generated by higher Gaudin Hamiltonians in the algebra of linear operators in the vector space \(\text{Hom}(V(\mu),V(\lambda^{(1)})\otimes\ldots\otimes V(\lambda^{(n)}))\) is isomorphic to the coordinate ring of the intersection of the following Schubert cells in the space $\mathbb{C}[x]_{<2n}$ of polynomials of the degree smaller than $2n$. Namely, to any $z_i$ one assigns the full flag $(x-z_i)^{2n-1}\mathbb{C}[x]_{<2n}\subset (x-z_i)^{2n-2}\mathbb{C}[x]_{<2n}\subset (x-z_i)^{2n-3}\mathbb{C}[x]_{<2n}\subset\ldots \mathbb{C}[x]_{<2n}$ and takes the Schubert variety with respect to this flag corresponding to the one-box Young diagram $(1)$. Next, one takes the Schubert variety corresponding to the partition $\lambda^{(0)}$ with respect to the flag $\mathbb{C}[x]_{<1}\subset\mathbb{C}[x]_{<2}\subset\ldots\subset\mathbb{C}[x]_{<2n}$. This is an isomorphism of schemes over the parameter space of the Gaudin Hamiltonians, i.e. over $\mathbb{C}^n\setminus\bigcup\limits_{i\ne j}\{z_i=z_j\}$. 
    
    So, the parameter space for the Gaudin Hamiltonians gets embedded into the parameter space for $n+1$ full flags in the $2n$-dimensional space in such a way that the solutions of Bethe ansatz for the Gaudin Hamiltonians get identified with the solutions to the above Schubert problem. This means that the Galois group of the Gaudin Bethe ansatz equations is a subgroup of the Galois group of the Schubert problem. Finally, by Corollary~\ref{co:Bethe-ansatz}, the latter is at least the alternating group. Hence the assertion.   
\end{proof}

\section{$S_N$ or $A_N$?}\label{se:further}

The numerical experiments show that the image of $C_n$ in the permutation group $S_N$ of $\syt(\lambda)$ is more frequently $A_N$ as $n$ gets bigger. On the other hand, there is no number of boxes past which for every non self-transpose and non hook-shaped $\lambda$ the cactus group has the image $A_N$. For example, the following infinite family of non-generic shapes always has $S_N$ as the image of $C_n$.

\begin{example}
    Consider the partition with $2^n+1$ boxes in the first row and 3 boxes in the second row. The non-fixed points of the operation $t_{n-1}$ occur precisely when $n-1$ and $n$ are in the only two corners. Hence, the number of swaps $t_{n-1}$ creates is the number of standard Young tableaux for the partition $(2^n,2)$, which is always odd.
\end{example}

However, for all Young diagrams that are not too wide and not too tall, we always have $A_N$. We thank Yakob Kahane for the idea and the key reference.

For a partition $\lambda$, let $H(\lambda)$ be the set of all hooks of $\lambda$, and $H_t(\lambda)$ be the number of hooks of $\lambda$ divisible by $t\in \mathbb{N}$.

\begin{thm}[Han, \cite{Han2010}, Corollary 5.1]
For a given shape $\lambda$, and $t$ a positive integer, we have
\[\sum_{\lambda} x^{|\lambda|} y^{H_t(\lambda)}
=
\prod_{k\ge 1}
\frac{(1-x^{tk})^{t}}{\bigl(1-(y x^{t})^{k}\bigr)^{t}\,(1-x^{k})}.
\]
\end{thm}

\begin{cor}\label{co:ht}
For all shapes $\lambda$, $H_t(\lambda)\le \left\lfloor \frac{|\lambda|}{t}\right\rfloor$.
\end{cor}

\begin{proof}
The RHS of the above formula is a power series of $x$ and $yx^t$,
so the power of $x$ in each monomial is greater than $t$ times the power of $y$.
\end{proof}

This implies the following:

\begin{prop} \label{pr:|syt|=even}
    Let $\lambda$ be a partition of $n$, and $m$ the integer such that $2^m\leq n <2^{m+1}$. If $\lambda$ fits into a $2^{m-1}\times 2^{m-1}$ grid, then $|\syt(\lambda)|$ is even.
\end{prop}

\begin{proof}
For any integer $k$, denote by $v_2(k)$ the maximal such that $k$ is divisible by $2^m$. Then, by the hook length formula, for any $\lambda\vdash n$, we have:
\begin{align*}
v_2(|\syt(\lambda)|)
&= v_2\!\left(\frac{n!}{\prod_{h\in H(\lambda)} |h|!}\right) \\
&= \sum_{k\ge 1}\left\lfloor \frac{n}{2^k}\right\rfloor
\;-\; \sum_{h\in H(\lambda)} v_2(|h|) \\
&= \sum_{k\ge 1}\left\lfloor \frac{n}{2^k}\right\rfloor
\;-\; \sum_{i\ge 1} i\cdot\bigl(H_{2^i}-H_{2^{i+1}}\bigr) \\
&= \sum_{k\ge 1}\left(\left\lfloor \frac{n}{2^k}\right\rfloor - H_{2^k}(\lambda)\right).
\end{align*}
From Corollary~\ref{co:ht}, we see that each term of the above sum is non-negative. However, since $\lambda$ fits into a $2^{m-1}\times 2^{m-1}$ grid, there are no hooks in $\lambda$ of size greater or
equal $2^m$.
Hence,
\[
\left\lfloor \frac{n}{2^m}\right\rfloor
- H_{2^{\lfloor\log_2(n)\rfloor}}(\lambda)
= 1 - 0 = 1.
\]
Therefore $v_2(|\syt(\lambda)|)\ge 1$.
\end{proof}

This implies the following:

\begin{cor}\label{co:n/8}
    If $\lambda$ fits into a $\left\lceil\frac{n}8\right\rceil\times \left\lceil\frac{n}8\right\rceil$ grid, then the image of $C_n$ is contained in $A_N$.
\end{cor}

\begin{proof}
 The group $C_n$ can be generated by the elements $s_{[1i]}$ with $i\le\frac{n}{2}$ and $t_j$ with $j\ge\frac{n}{2}$. Note that $s_{[1i]}$ is conjugate to $s_{[n-i+1,n]}$ in the cactus group, so, by Proposition~\ref{pr:Schutzenberger} all these generators are at least conjugate to the ones preserving subtableau formed by the boxes that contain the numbers up to $\frac{n}{2}-1$. Next, for any standard Young tableau $T$, its subtableau $T'$ formed by the boxes that contain the numbers up to $\frac{n}{2}-1$ satisfies the condition of Proposition~\ref{pr:|syt|=even}. So, the number of independent transpositions in any of the above generators is a multiple of the number of such $T'$ -- which is even! Thus all the generators of the image of $C_n$ are even permutations of $\syt(\lambda)$.
\end{proof}

\noindent \textbf{Remark.} We define \emph{generic} partitions as those partitions that fit into a $2\sqrt2\cdot \sqrt{|\lambda|}$ by $2\sqrt2\cdot \sqrt{|\lambda|}$ grid. This definition of generic accounts for all randomly chosen partitions of $n$ in the limit according to the Plancherel measure described by Vershik  and Kerov in \cite{Vershik_Kerov_1977}. We have checked by computer how frequent the image of $C_n$ is $A_N$ or $S_N$ for all partitions up to 52 boxes, and \Cref{S_N A_N plot} in the Appendix shows the distribution for all partitions and for generic partitions. The data suggests that the number of images in the alternating group will eventually dominate the number in the full group $S_N$, and indeed \Cref{co:n/8} shows this is true. In particular, when $|\lambda|$ exceeds 504, \Cref{co:n/8} tells us that $A_N$ is the only possible image for generic partitions.

\medskip

\noindent \textbf{Remark.} The computer experiment shows that the condition that $\lambda$ fits into a $2^{m-1}\times 2^{m-1}$ grid in Proposition~\ref{pr:|syt|=even} is sharp; see Figure~\ref{evenness plot} in the Appendix.

\printbibliography

\bigskip

\noindent\footnotesize{{\bf Sophia Liao}
 \\Harvard University, Cambridge MA, USA\\
{\tt sliao@college.harvard.edu}} \\

\noindent\footnotesize{{\bf Leonid Rybnikov} \\
Department of Mathematics and Statistics,
University of Montreal, Montreal QC, Canada\\
{\tt leonid.rybnikov@umontreal.ca}}

\newpage
\appendix \section{Computational Results}

\begin{figure}[ht]
    \centering
    \includegraphics[width=0.95\linewidth]{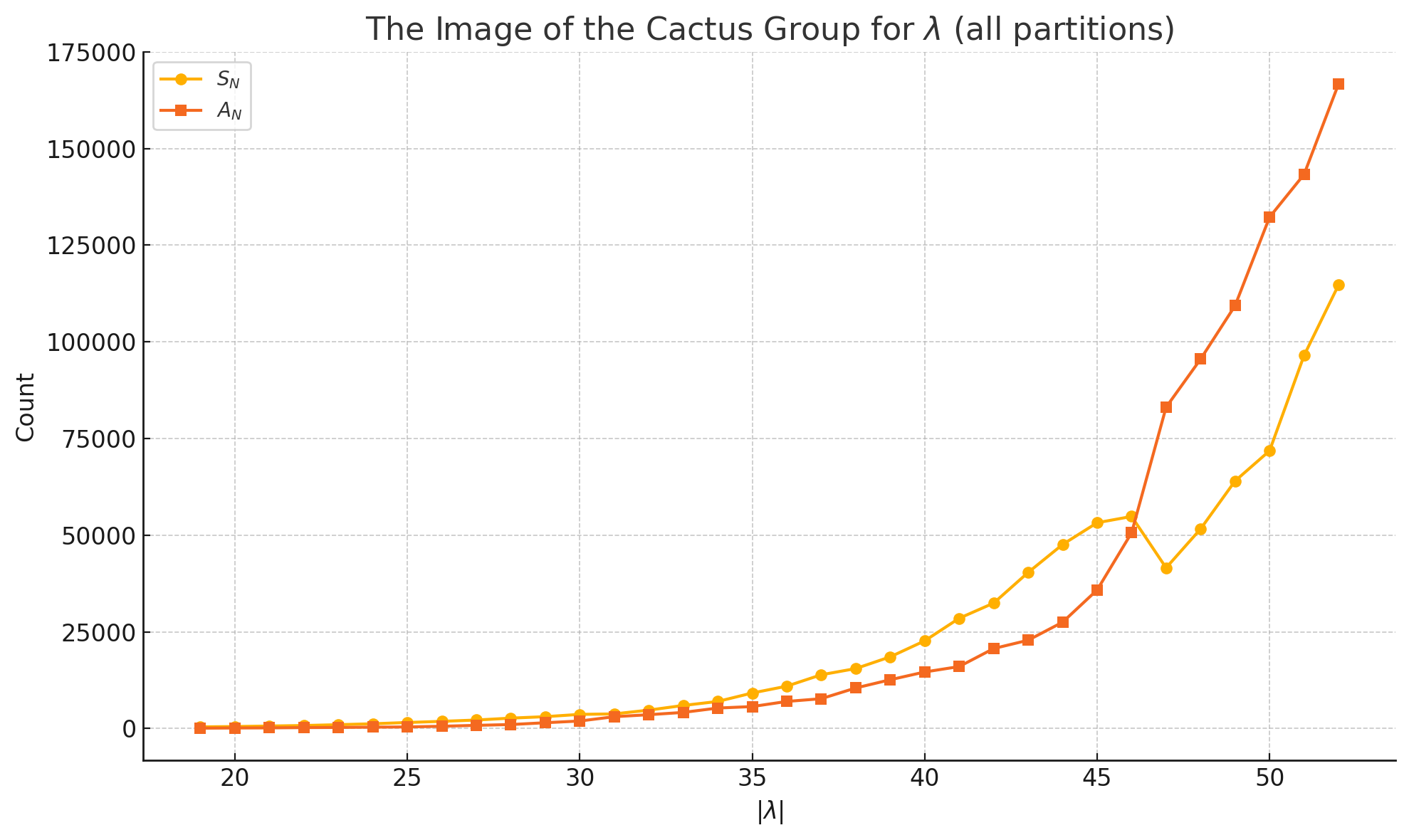}
    \includegraphics[width=0.95\linewidth]{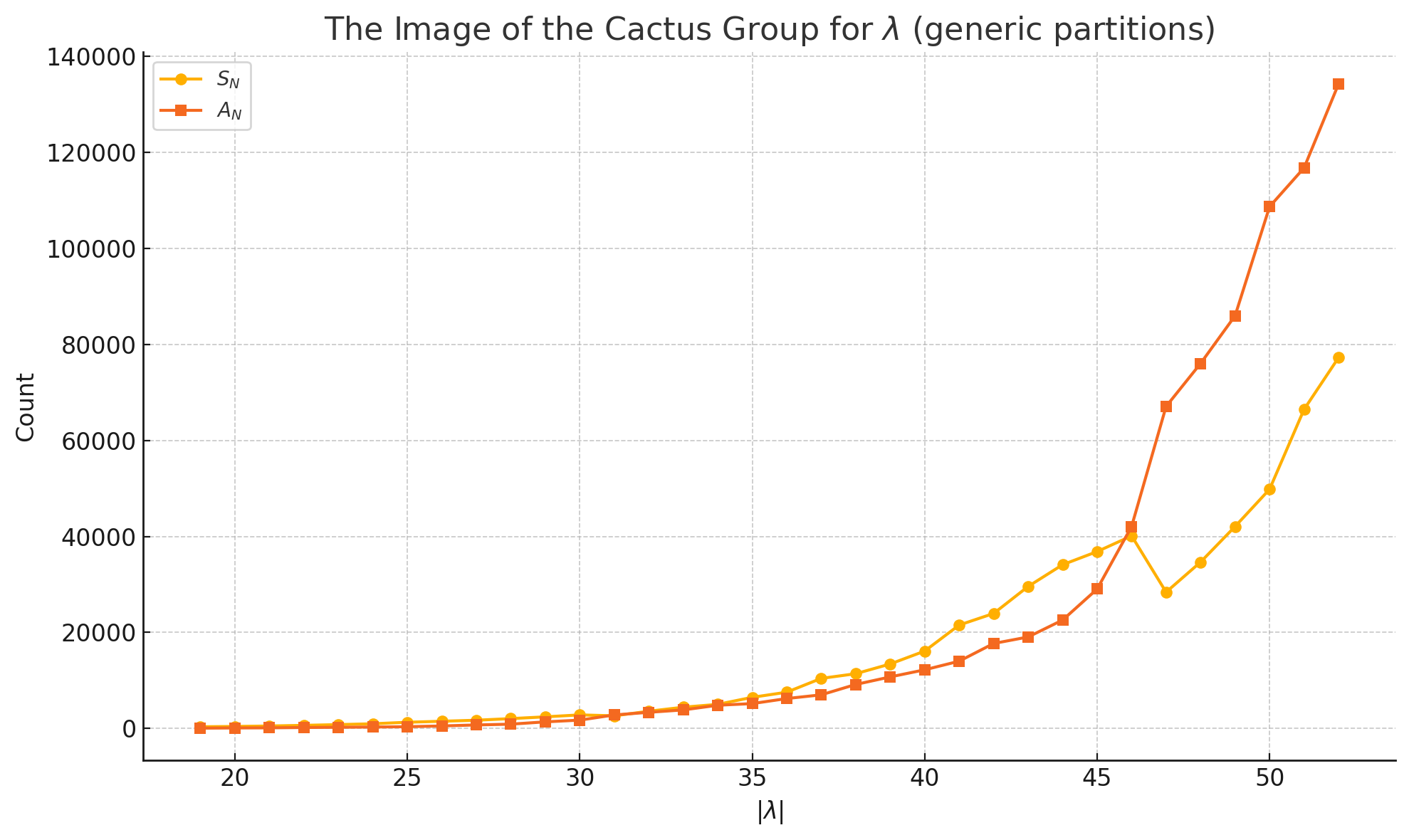}
    \caption{The number of partitions $\lambda$ such that the image of $\syt(\lambda)$ under the cactus group is $S_N$ or $A_N$. The first graph considers all possible partitions, whereas the second only considers generic partitions (those that fit into a $2\sqrt2\cdot \sqrt{|\lambda|}$ by $2\sqrt2\cdot \sqrt{|\lambda|}$ grid).}
    \label{S_N A_N plot}
\end{figure}

\begin{figure}[ht]
    \centering
    \includegraphics[width=0.95\linewidth]{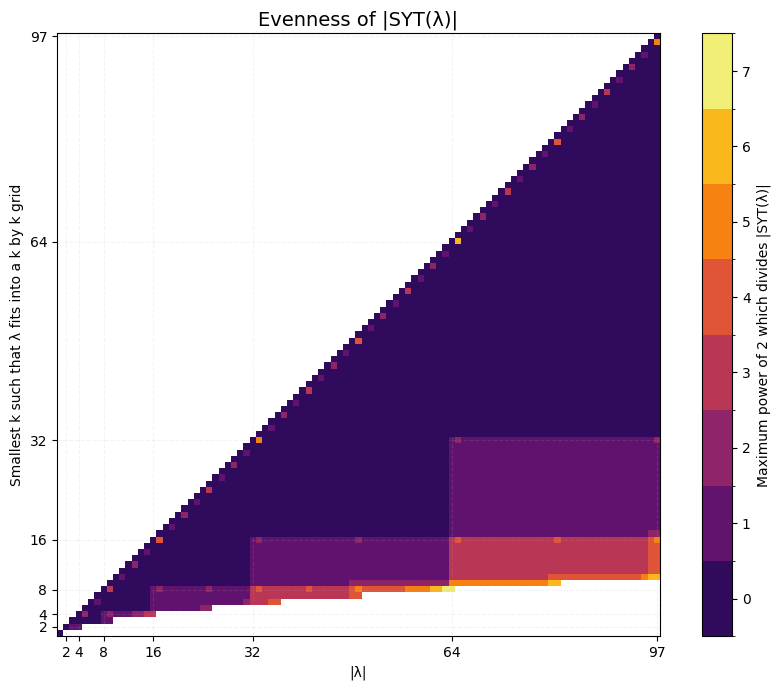}
    \caption{Describes how even $|\syt(\lambda)|$ is for two parameters of $\lambda$: the number of boxes and the maximum of its height and width. Each square on the plot represents all partitions of a given size ($x$-axis) and a given $k$ where $\lambda\subseteq k\times k$ but $\lambda\not\subseteq (k-1)\times(k-1)$ ($y$-axis). The square is assigned a color based on the maximum power of 2 which divides $|\syt(\lambda)|$ for all such $\lambda$ considered by these two criteria. For example, every partition of 32 boxes which fit into a $16\times 16$ grid have an even number of standard Young tableaux, but there is at least one 32-box partition whose height or width is 17 that has an odd number of standard Young tableaux.}
    \label{evenness plot}
\end{figure}

\end{document}